\documentclass[VANCOUVER,STIX1COL]{WileyNJD-v2}
\usepackage{etex}
\usepackage{enumerate}
\usepackage{float}
\usepackage{epic,eepic,graphicx}
\usepackage{multicol}
\usepackage{subcaption}
\usepackage{epstopdf}
\usepackage{amsmath,amssymb}
\usepackage{pgfplots}
\usepackage{bbm}
\usepackage{booktabs}
\usepackage{tabularx}
\usepackage{color}
\usepackage{colortbl}
\usepackage{pgfplotstable}
\usepackage{tikz}
\usepackage{rotating}
\usepackage{multirow,bigdelim}
\usepackage{verbatim}
\usetikzlibrary{pgfplots.groupplots}
\usetikzlibrary{matrix,positioning}
\usetikzlibrary{external}
\usetikzlibrary{patterns}
\tikzsetexternalprefix{Tikz/}
\usetikzlibrary{shapes}
\usetikzlibrary{arrows}
\usetikzlibrary{calc}  
    \usetikzlibrary{
	pgfplots.colorbrewer,
}
\pgfplotsset{compat=newest}
\pgfplotsset{
	cycle list/.define={my marks}{
		every mark/.append style={solid,fill=\pgfkeysvalueof{/pgfplots/mark list fill}},mark=*\\
		every mark/.append style={solid,fill=\pgfkeysvalueof{/pgfplots/mark list fill}},mark=square*\\
		every mark/.append style={solid,fill=\pgfkeysvalueof{/pgfplots/mark list fill}},mark=triangle*\\
		every mark/.append style={solid,fill=\pgfkeysvalueof{/pgfplots/mark list fill}},mark=diamond*\\
	},
}

\pgfplotscreateplotcyclelist{clusters}{%
	blue,thick,every mark/.append style={solid},mark=*\\%
	red,thick,every mark/.append style={solid},dash dot,mark=square*\\%
	black!40!green,thick,every mark/.append style={solid},mark=triangle*\\%
	black!20!cyan,thick,every mark/.append style={solid},dashed,mark=diamond*\\%
	black!20!magenta,thick,every mark/.append style={solid},mark=pentagon*\\%
	black!40!blue,densely dotted,every mark/.append style={fill=blue!80!black},mark=*\\%
	black!20!blue,dashed,every mark/.append style={fill=blue!80!black},mark=*\\%
	red,every mark/.append style={fill=blue!80!black},mark=*\\%
	black!40!red,densely dotted,every mark/.append style={fill=blue!80!black},mark=*\\%
	black!20!red,dashed,every mark/.append style={fill=blue!80!black},mark=*\\%
	green,every mark/.append style={fill=blue!80!black},mark=*\\%
	black!40!green,densely dotted,every mark/.append style={fill=blue!80!black},mark=*\\%
	black!20!green,dashed,every mark/.append style={fill=blue!80!black},mark=*\\%
	cyan,every mark/.append style={fill=blue!80!black},mark=*\\%
	black!40!cyan,densely dotted,every mark/.append style={fill=blue!80!black},mark=*\\%
	black!20!cyan,dashed,every mark/.append style={fill=blue!80!black},mark=*\\%
}
\pgfplotscreateplotcyclelist{smoothers}{%
		blue,thick,,every mark/.append style={fill=blue!80!black},mark=*\\%
	red,thick,dash dot,every mark/.append style={fill=blue!80!black},mark=*\\%
	black!40!green,thick,densely dotted,every mark/.append style={fill=blue!80!black},mark=*\\%
	black!20!cyan,thick,dashed,every mark/.append style={fill=blue!80!black},mark=*\\%
	blue,every mark/.append style={fill=blue!80!black},mark=*\\%
	red,every mark/.append style={fill=blue!80!black},mark=*\\%
	green,every mark/.append style={fill=blue!80!black},mark=*\\%
	black!40!blue,densely dotted,every mark/.append style={fill=blue!80!black},mark=*\\%
	black!40!red,densely dotted,every mark/.append style={fill=blue!80!black},mark=*\\%
	black!40!green,densely dotted,every mark/.append style={fill=blue!80!black},mark=*\\%
	black!20!blue,dashed,every mark/.append style={fill=blue!80!black},mark=*\\%
	black!20!red,dashed,every mark/.append style={fill=blue!80!black},mark=*\\%
	black!20!green,dashed,every mark/.append style={fill=blue!80!black},mark=*\\%
}

\usepackage{aeguill}

\tikzset{class or interface/.style={%
		draw,%
		shape=rectangle split,%
		rectangle split parts=3,%
		rectangle split part align={center,left,left},%
		rectangle split part fill={#1!30!white,#1!20!white,#1!10!white},
		every two node part/.style={align=center, font=\ttfamily},
		every three node part/.style={align=left, font=\ttfamily}, 
		node distance=2
	},
	interface/.style={
		class or interface=blue!80!black,
		every one node part/.style={align=center, font=\ttfamily\itshape}
	},                       
	concrete class/.style={
		class or interface=green!70!black,
		every one node part/.style={align=center, font=\ttfamily\bfseries}
	}
}

\pgfplotsset{
	log x ticks with fixed point/.style={
		xticklabel={
			\pgfkeys{/pgf/fpu=true}
			\pgfmathparse{exp(\tick)}%
			\pgfmathprintnumber[fixed relative, precision=3]{\pgfmathresult}
			\pgfkeys{/pgf/fpu=false}
		}
	},
	log y ticks with fixed point/.style={
		yticklabel={
			\pgfkeys{/pgf/fpu=true}
			\pgfmathparse{exp(\tick)}%
			\pgfmathprintnumber[fixed relative, precision=3]{\pgfmathresult}
			\pgfkeys{/pgf/fpu=false}
		}
	}
}

\articletype{Article Type}%

\received{26 April 2016}
\revised{6 June 2016}
\accepted{6 June 2016}

\def\arraystretch{0.75} 
\begin{document}
	
	\title{Two-level preconditioning for Ridge Regression\protect\thanks{This work was supported Research Foundation - Flanders (FWO, No. G079016).}}
	
	\author[1,2]{J. Tavernier*}
	
	\author[2]{J. Simm}
	
	\author[1]{K. Meerbergen}
	
	\author[2]{Y. Moreau}
	
	\authormark{J. Tavernier \textsc{et al}}

	\address[1]{\orgdiv{Departement of Computer Science, NUMA},  \orgname{KU Leuven}, \country{Belgium}}
	\address[2]{\orgdiv{Department of Electrical Engineering, ESAT - STADIUS}, \orgname{KU Leuven}, \country{Belgium}}

	\corres{*Joris Tavernier, \url{https://people.cs.kuleuven.be/~joris.tavernier/},  \email{Joris.Tavernier@kuleuven.be}.}
	
	
	\abstract[Summary]{Solving linear systems is often the computational bottleneck in real-life problems. Iterative solvers are the only option due to the complexity of direct algorithms or because the system matrix is not explicitly known. Here, we develop a two-level preconditioner for regularized least squares linear systems involving a feature or data matrix.  Variants of this linear system may appear in machine learning applications, such as ridge regression, logistic regression, support vector machines and Bayesian regression. We use clustering algorithms to create a coarser level that preserves the principal components of the covariance or Gram matrix. This coarser level approximates the dominant eigenvectors and is used to build a subspace preconditioner accelerating the Conjugate Gradient method. We observed speed-ups for artificial and real-life data. }
	
	
	
	\maketitle
	

	\renewcommand\algorithmiccomment[1]{\hfill{\color{green!40!black}$\blacktriangleright$ #1}}
	
	\maketitle

	\section{Introduction} 
	The computational bottleneck in Machine Learning applications is often solving a linear system. These systems are usually solved iteratively due to the complexity of direct algorithms. Here we focus on problems from machine learning with variations of a least squares linear system leading to the regularized normal equations \begin{equation}(X^TX+\beta I)\textbf{w}=X^T\textbf{b} \label{eq:xtx=xtb}\end{equation}
	with $\beta$ a given regularization parameter. The feature matrix $X \in \mathbb{R}^{N\times F}$ consists of $N$ data samples with $F$ features. This matrix is often high-dimensional and sparse. Hence, Krylov subspace methods are appealing, since they only require the matrix-vector multiplication and not the matrix itself. For arbitrary sparse feature matrices $X$, the matrix product $X^TX$ in (\ref{eq:xtx=xtb}) is likely to be much less sparse than $X$. We focus on underdeterminded systems with $N<F$, referred to as small sample problems. This results in ill-posed linear systems with a non trivial null space $\mathcal{N}(X^TX)\ne\emptyset$. 
	
	The number of iterations required to reach a certain tolerance for Krylov subspace methods depends on the spectrum of the matrix, more specifically the distribution of the eigenvalues  \cite{demmel1997applied,trefethen1997numerical}. The condition number of the matrix is often an adequate indication. The art of preconditioning aims at reducing the number of iterations by transforming the problem to an easier one. One simple approach is to divide the matrix by its main diagonal and is called Jacobi preconditioning \cite{demmel1997applied,vanderSluis1969Condition}. Diagonal or Jacobi preconditioning is often very efficient if the matrix is diagonally dominant. Other existing preconditioning techniques such as incomplete Cholesky factorization \cite{KERSHAW197843incomplete,Lin1999Incomplete} or incomplete LU factorization \cite{Saad1994ILUT} are useless since the matrix product in our linear system is often close to dense. It is possible that the covariance matrix or Gram matrix $X^TX$ follows some special structure for which the inverse is sparse. In this case, the Sparse Approximate Inverse \cite{Benzi1996Approximate} can be computed based on this sparse structure. However, in general the inverse of the covariance matrix is dense. 
	
	For the solution of partial differential equations (PDEs), geometric Multigrid \cite{briggs2000multigrid} and its algebraic extensions are often successful \cite{ruge1987algebraic,stuben2001review,vassilevski2008multilevel}. The crucial concept is that the eigenvectors that slow down the convergence of iterative methods, can be accurately represented on a coarser grid. The components of these eigenvectors are then eliminated on coarser grids \cite{brezina2005adaptive,brezina2006adaptive,maclachlan2006adaptive}. Currently, the most generic preconditioners are variants of bootstrap AMG \cite{brandt2011bootstrap,Ambra2019Adaptive}. Bootstrap AMG designs the hierarchy of grids such that certain set of slow converging eigenvectors are exactly interpolated from the coarser grids. Bootstrap additionally exploits the existing hierarchy of grids to improve this set of slow converging eigenvectors. In this way, the hierarchy of grids is recursively improved. 
	
	The multilevel idea was also developed for the Fredholm integral equations and the resulting ill-posed problems \cite{reichel2010cascadic}. The unregularized problem is solved by applying the Conjugate Gradient method (CG) for the normal equations on different levels. Regularization is applied by restricting the number of iterations on the different levels. In addition, principal eigenvectors can be used for preconditioning LSQR for ill-posed problems \cite{jacobsen2003subspace,paige1982lsqr}. For PDEs, multilevel preconditioning exploits the underlying idea that the operator is based on a discretized grid, often resulting in M-matrices. As for the Fredholm integral equations, the operator is based on a discretization of continuous functions on a grid. Data matrices however do not have this property.
	
	Randomized algorithms such as stochastic gradient descent are popular for large-scale optimization problems in data sciences \cite{StochasticRobbins,BottouSGD2010,NIPS2013_4937}. These randomized algorithms have been used effectively for least-squares solves of sparse, well-conditioned matrices \cite{Kaczmarz2013Zouzias}. For an overview of randomized iterative algorithms for linear systems, see \cite{Randomized2015Gower} and the references therein. 
		
	In machine learning, several linear systems of the form (\ref{eq:xtx=xtb}) may have to be solved for one application. Firstly, the regularization parameter $\beta$ in (\ref{eq:xtx=xtb}) is often unknown and several values are tried in a cross-validation process. Secondly, different right-hand sides are often available or the linear system to be solved is part of a sampling algorithm. This means that even high computational costs for the construction of the preconditioner is often negligible compared to the many solves. 
	
	For applications from data science, state-of-the-art multilevel methods can not be used in their standard form. As opposed to grid based operators in PDEs, the matrices consist of samples and features and a hierarchy has to be designed on the data level. The resulting linear operator is a product of matrices and often dense. This means that the operator itself is not always available due to memory restrictions. There is no natural way to design the level hierarchy and the system matrix may not be available to design the coarse level directly from the operator. We thus propose to build a multilevel preconditioner based on the Multigrid idea by clustering the features of the data matrix, resulting in coarser levels.  Figure \ref{fig:lvlhierarchy} shows an example for dense matrices with clustering for both rows and columns. This can be done such that the principal eigenvectors are represented on the coarser levels. In practice, it is however more efficient to directly project to the coarsest level and only use the finest and the coarsest level and design a two-level preconditioner. 

\begin{figure*}[ht]
	\tikzsetnextfilename{lvlhierarchy}
	\begin{subfigure}{\textwidth}
		\centering
		\resizebox{\linewidth}{!}{\begin{tikzpicture}
\node at (0,-2.5) {Level 0};
\node at (5,-2.5) {Level 1};
\node at (10,-2.5) {Level 2};
\draw[step=0.5,lightgray] (-2,-2) grid (2,2);
 \foreach \x in {0,...,8}
\foreach \y in {0,...,8}
{
	\fill (\x/2-2,\y/2-2) circle (2.5pt);
}

\draw[step=1,lightgray] (3,-2) grid (7,2);
\foreach \x in {0,...,4}
\foreach \y in {0,...,4}
{
	\fill (\x+3,\y-2) circle (2.5pt);
}

\draw[step=2,lightgray] (8,-2) grid (12,2);
\foreach \x in {0,...,2}
\foreach \y in {0,...,2}
{
	\fill (2*\x+8,2*\y-2) circle (2.5pt);
}
\end{tikzpicture}}
		\caption{Hierarchy of Grids}
		\label{fig:grid_levels}
	\end{subfigure}
	\begin{subfigure}{\textwidth}
		\centering
		\resizebox{\linewidth}{!}{\begin{tikzpicture}
\node at (-1,-2.3) {$X^T_0$};
\node at (2.2,-0.3) {$X_0$};
\node at (5.5,-0.3) {$X^T_1$};
\node at (7.5,0.7) {$X_1$};
\node at (9.3,0.7) {$X^T_2$};
\node at (10.5,1.2) {$X_2$};
\node at (0.5,2.3) {Level 0};
\node at (6.5,2.3) {Level 1};
\node at (10,2.3) {Level 2};
\draw[step=0.5] (-2,-2) grid (0,2);
\draw[step=0.5] (0.499,0) grid (4.5,2);
\draw[step=0.5] (4.99,0) grid (6,2);
\draw[step=0.5] (6.499,0.99) grid (8.5,2);
\draw[step=0.5] (8.99,0.99) grid (9.5,2);
\draw[step=0.5] (9.99,1.499) grid (11,2);
\end{tikzpicture}}
		\caption{Hierarchy of dense data matrices for $X^TX$}
		\label{fig:mat_levels}
	\end{subfigure}
	\caption{Examples of the level hierarchies in PDEs and Data. Figure \ref{fig:grid_levels} shows the hierarchy of grids for the discretization of a PDE. Figure \ref{fig:mat_levels} shows the hierarchy of levels for $X^tX$ using a data matrix $X$ and clustering both the rows and columns.}
	\label{fig:lvlhierarchy}
\end{figure*} 

	The paper is organized as follows. Section \ref{sect:problem} describes different algorithms from machine learning and identifies the linear systems to be solved. Next, Section \ref{sect:two-level} describes the two-level preconditioner. In Section \ref{sect:Experiments}, the numerical results are presented. Finally, conclusions are given in Section \ref{sect:conclusion}. 
	\section{Applications}\label{sect:problem}
	Data in machine learning consists of data samples and their features. Feature matrices can take any form, from dense or sparse and binary or real valued. The sample characteristics $\textbf{x}_i$ are collected in a feature matrix  $X\in \mathbb{R}^{N\times F}$ with $N$ the number of samples or data points and $F$ the number of features. Using the collected feature matrix, one tries to predict certain properties for new and unknown samples. These properties can be real valued and then the task at hand is regression or they can be limited to certain classes and then the task is classification. In this section, we will describe some techniques from machine learning where the linear system (\ref{eq:xtx=xtb}) is the main computational bottleneck, with $\beta$ the regularization parameter. The techniques described here are mainly intended as illustrations and a comprehensive understanding is not required for the remaining sections. These applications require several linear systems with different right hand sides and justify the cost of more computational intensive preconditioners such as the preconditioner proposed in Section \ref{sect:two-level}.  
	\subsection{Ridge Regression}
	Suppose we have more samples than features resulting in an overdetermined system $$\textbf{b}\label{eq:lm}=X\textbf{w}+\textbf{e}$$ with $\textbf{b}$ the property of interest for the samples \cite{bishop2006pattern}. The ridge regression solution is determined by minimizing the squared error on $\textbf{b}$ ($\textbf{e}=X\textbf{w}-\textbf{b}$). Adding regularization to the solution vector $\textbf{w}$, the optimization problem is given by  $$\min_{\textbf{w}}\textbf{e}^T\textbf{e}+\beta \textbf{w}^T\textbf{w}$$ and the solution is \begin{equation}
	\textbf{w}=(X^TX+\beta I)^{-1}X^T\textbf{b}. \label{eq:rr}
	\end{equation} This technique is known in the field of numerical linear algebra as Tikhonov regularization for linear systems. The value of the regularization parameter $\beta$ is unknown and several values are tried in a cross-validation process.  
	
	\subsection{Bayesian regression}
	In ordinary ridge regression \cite{bishop2006pattern,gelman2013bayesian} the residuals $\textbf{e}$  are assumed uncorrelated and normally distributed in \eqref{eq:lm}, e.g. $\textbf{e}\sim \mathcal{N}(\textbf{0},\tau^{-1}I_N)$ with $I_N$ the identity matrix of size $\mathbb{R}^{N\times N}$ and $\tau > 0$ the precision parameter. In Bayesian learning, the observations $\textbf{b}$ are considered probabilistic \begin{equation*}
	p(\textbf{b}|X,\textbf{w})=\prod_{i=1}^N \mathcal{N}(y_i|{\textbf{x}}_i{\boldsymbol{\textbf{w}}},\tau^{-1})
	\end{equation*} with $\textbf{x}_i\in \mathbb{R}^{1\times F}$ the features of sample $i$. The conditional posterior distribution of $\textbf{w}$ given $\tau,X,\beta$ and $\textbf{b}$ is then \begin{equation*}
	p(\textbf{w}|\tau,\textbf{b},X) \sim \mathcal{N}((X^TX+\beta I_F)^{-1}X^T\textbf{b},((X^TX+\beta I_F)\tau)^{-1})
	\end{equation*}
	with $\beta$ a regularization parameter. Computing the covariance matrix $(X^TX+\beta I_F)\tau$ is inconvenient and often computationally infeasible. Simm et al. \cite{Simm2017Macau} however describe a Gibbs sampler for Bayesian regression and using their noise injection trick, a sample of $\textbf{w}$ can be taken by solving \begin{equation}
	\Biggl({X}^T{X}+	\frac{\lambda}{\tau}I_F\Biggr){\boldsymbol{\textbf{w}}}= {X}^T(\textbf{y}+\textbf{e}_1)+\frac{\textbf{e}_2}{\tau} \label{eq:ni_xtx}
	\end{equation}
	with $\textbf{e}_1\sim \mathcal{N}(\textbf{0},\tau^{-1}I_N)$ and $\textbf{e}_2\sim \mathcal{N}(\textbf{0},\lambda I_F)$.  A zero mean and uncorrelated normal distribution is used as prior for $\textbf{w}$: \begin{equation*}
	p(\textbf{w},\lambda|\alpha,\beta)\sim \mathcal{N}(\textbf{w}|\textbf{0},\lambda^{-1}I_F)\mathcal{G}(\lambda|\alpha,\beta)
	\end{equation*}
	with $\mathcal{G}(\lambda|\alpha,\beta)$ the gamma distribution as conjugate prior for the precision parameter $\lambda>0$. The parameters $\lambda$ and $\tau$ are sampled and the fraction $\frac{\lambda}{\tau}$ thus changes each iteration. The noise injection trick was originally developed for Bayesian Probabilistic Matrix Factorization \cite{salakhutdinov2008bayesian} with side information as detailed in \cite{Simm2017Macau}.
	\subsection{Optimization by the Newton Method} For certain optimization problems $$ \min_{\textbf{w}}f(w) $$ the solution can be found by the Newton method which requires the inverse of the Hessian ($\nabla^2 f(\textbf{w}))$). The Newton approach takes, at iteration $k$, the step $\textbf{s}_k=-(\nabla^2 f(\textbf{w}_{k-1}))^{-1}	\nabla f(\textbf{w}_{k-1})$ and the iterative solution is updated as $\textbf{w}_{k}=\textbf{w}_{k-1}+\textbf{s}_k$ with $\nabla f(\textbf{w})$ the gradient. For larg\text{e-}scale data the inverse is not directly computed but a linear system is solved. Binary logistic regression \cite{hosmer2013applied} and L2 loss function Support Vector Machines \cite{Vapnik1999overview} for example can be solved by the Newton Method and the Hessians $$\nabla^2 f(\textbf{w})=X^TD(\textbf{w})X+\beta I$$ are variants of (\ref{eq:xtx=xtb}) with the system matrix a product and $D(\textbf{w})$ a diagonal matrix. More details can be found in the paper by Lin et al. \cite{Lin2008Trust} and the references therein. 
	
	\subsection{Linear discriminant analysis}
	Dimension reduction techniques such as Principal Component Analysis (PCA) or Linear discriminant analysis (LDA) try to reduce the number of variables used in further analysis such as classification \cite{bishop2006pattern}. LDA is a supervised method that finds a linear mapping that maximally separates the classes while minimizing the within class variance. Defining the total-class-scatter matrix as $S_T=X^TX$ and the between-class-scatter matrix as $S_B$ the LDA optimization problem is $$\max_{\textbf{w}\in\mathbb{R}^{N}} \frac{\textbf{w}^TS_B\textbf{w}}{\textbf{w}^TS_T\textbf{w}}$$ 
	and can be solved by finding the eigenvectors corresponding to the largest eigenvalues of $$S_B\textbf{w}_i=\lambda_i(S_T+\beta I)\textbf{w}_i,$$ with $\beta$ the regularization parameter. The low-rank and structure of the between-class-scatter matrix can be exploited and the main computational bottleneck is solving the linear system \begin{equation}
	(X^TX+\beta I)\textbf{w}=X^T\textbf{b} \label{eq:lda}
	\end{equation} with $\textbf{b}$ determined by the sample labels as detailed in \cite{cai2008srda}.    
	
	\section{Two-level preconditioner}	\label{sect:two-level}
	The general idea of two-level preconditioning is to precondition a linear system from the solution of a linear system with less variables. This smaller linear system is obtained by projecting the large linear system on a subspace. Geometric Multigrid is used for the solution of elliptic partial differential equations. The discrete problem is solved hierarchically by approximating the original discretized PDE by coarser discretizations of the PDE \cite{briggs2000multigrid}. The Multigrid principles were extended to Algebraic Multigrid using only the coefficients and sparsity pattern of the matrix instead of the geometry of the problem \cite{ruge1987algebraic,stuben2001review,vassilevski2008multilevel}.   
	\subsection{The idea of two-level iterations}
	The use of coarse grids for the solution of PDEs is motivated as follows. Many iterative schemes have poor convergence for PDEs. For a class of solvers, e.g. Gauss-Seidel iteration, the low frequency and near kernel components of the solution converge very slowly. These solvers are called smoothers because the remaining error of the solution is typically smooth. The coarse grids are used to remove the low frequency components of the error in a more efficient way. When the PDE is discretized on a coarse grid, the low frequency components of the solution on the fine grid and thus the error can be approximated well on to the coarse grid. Generally, a two-level solver consists of the following steps. A small number of iterations is performed on the fine level and a smooth error remains. The residual is mapped on the coarse level and this is called restriction. The restriction operator $R$ maps a (long) vector on the fine grid to a (short) vector on the coarse grid. The linear system on the coarse grid is solved exactly. The coarse solution is mapped back to the fine grid using a prolongation operator $P$. This uses interpolation to identify all the unknowns on the fine grid and leads to a solution on the fine grid that mainly has high frequency components in the error. These components can then be further reduced by additional smoothing steps. 
	
	For symmetric matrices, the restriction operator $R$ is defined as the transpose of the interpolation operator $P$. Using the Galerkin approach the coarse level is defined by $A_1=PA_0P^T$. Assuming one step of Richardson's iteration, with $\omega$ a nonnegative scalar, as smoother \cite{Saad2003Iterative}, the resulting iteration matrix is $$T=(I-\omega A_0)(I-PA_1^{-1}P^TA_0)(I-\omega A_0)$$ where the first factor is called post-smoothing, the third is pr\text{e-}smoothing and the middle term is the coars\text{e-}grid correction. A more detailed explanation of (algebraic) Multigrid can be found in  \cite{vassilevski2008multilevel} and the references therein.
	\subsection{Data matrices}
	We want to adopt the idea of approximating error components on the coarse level and prolonging the coarse solution to the fine level. For matrices resulting from PDEs, the effect of aggregating elements is well known. In our case the matrix product of interest is $X^TX$ with $X$ a data matrix and this matrix product defines the covariance matrix.
		
	The eigenvectors of the covariance matrix have a specific physical interpretation. A technique called Principal Component Analysis (PCA) finds the eigenvectors corresponding to the largest eigenvalues of the covariance matrix. PCA is in essence a projection that maximizes the total variance \cite{bishop2006pattern}. PCA actually finds the directions or principal components that projects the data on a subspace maximally preserving the variance of the data. For small sample size problems, the number of features is larger than the number of samples, resulting in a non trivial nullspace of $X^TX$. These eigenvectors are actually noise and contain no data information.
	
	A widely used technique in machine learning is clustering. These clustering algorithms find subsets of similar samples or features based on a distance or similarity measure. Ideally, this results in subsets with small variance within one specific subset. Clustered subsets of features can be used to create aggregate features for the coarse level, preserving highest variance in the data set. Since highest variance is preserved in the coarse level, the principal components can be approximated on the coarse level.
	
	Firstly, we propose to define an intuitive averaging restriction operator ($\tilde{P}^T$) based on the clusters and this operator will be improved later. Specifically, a coarse feature $\textbf{c}_\mathcal{S}$ consist of the average of the $n_\mathcal{S}$ features $X(:,i)$ in one cluster $\mathcal{S}$ and $i\in \mathcal{S}$: $$\textbf{c}_\mathcal{S}=1/n_\mathcal{S}\sum_{i\in\mathcal{S}} X(:,i).$$ The resulting restriction operator has the value $1/n_\mathcal{S}$ for each feature $i\in\mathcal{S}$ in the row corresponding to $\mathcal{S}$. Assuming that the matrix $X$ is ordered such that $X=[X_1, X_2, \dots, X_{F_C}]$ with $X_\mathcal{S}$ the features belonging to cluster $\mathcal{S}=1,...,F_C$ and $F_C$ the coarse feature dimension, the restriction operator is 
	\begin{equation}
		\tilde{P}^T=\begin{array}{cccccccccccc}	
		&\multicolumn{3}{c}{\overbrace{\hspace{50pt}}^{n_1}}&\multicolumn{3}{c}{\overbrace{\hspace{50pt}}^{n_2}}&\dots &\multicolumn{3}{c}{\overbrace{\hspace{55pt}}^{n_{F_C}}}& \\
		\ldelim[{5.3}{10pt}[] &\frac{1}{n_1} &\ldots &\frac{1}{n_1} & & & & & & & &\rdelim]{5.3}{10pt}[]  \\
		&  & & &\frac{1}{n_2} &\ldots  &\frac{1}{n_2} &  & & & &  \\
		&  & & & & & & \ddots& & & & \\
		& & & & & & & &\frac{1}{n_{F_C}} &\ldots &\frac{1}{n_{F_C}} & \\
		\end{array}
	\end{equation}

	Note that the product \begin{equation}
	\tilde{P}^T\tilde{P}=\begin{bmatrix}
	\frac{n_1}{n_1^2} & & & \\
	 &\frac{n_1}{n_2^2} & & \\
	 & &\ddots & \\
	 & & &\frac{n_{F_C}}{n_{F_C}^2} \\
	\end{bmatrix}
	\end{equation} and thus $\tilde{P}$ does not define a projection. The averaging interpolation operator can easily be adjusted such that the product $P^TP$ is the identity matrix by defining the coarse features as $$\textbf{c}_\mathcal{S}=1/\sqrt{n_\mathcal{S}}\sum_{i\in\mathcal{S}} X(:,i)$$ resulting in 	\begin{equation}
	P^T=\begin{array}{cccccccccccc}	
	&\multicolumn{3}{c}{\overbrace{\hspace{60pt}}^{n_1}}&\multicolumn{3}{c}{\overbrace{\hspace{60pt}}^{n_2}}&\dots &\multicolumn{3}{c}{\overbrace{\hspace{65pt}}^{n_{F_C}}}& \\
	\ldelim[{5.3}{10pt}[] &\frac{1}{\sqrt{n_1}} &\ldots &\frac{1}{\sqrt{n_1}} & & & & & & & &\rdelim]{5.3}{10pt}[]  \\
	&  & & &\frac{1}{\sqrt{n_2}} &\ldots  &\frac{1}{\sqrt{n_2}} &  & & & &  \\
	&  & & & & & & \ddots& & & & \\
	& & & & & & & &\frac{1}{\sqrt{n_{F_C}}} &\ldots &\frac{1}{\sqrt{n_{F_C}}} & \\
	\end{array} \label{eq:prolong}
	\end{equation} Using the Galerkin approach \cite{vassilevski2008multilevel} the coarse linear system becomes \begin{eqnarray*}
	P^T(X^TX+\beta I)P &=	P^TX^TXP+P^T\beta IP\\
	&= X_c^TX_c+\beta \underbrace{P^TP}_{=I}
\end{eqnarray*}
	with $X_c=XP$ the matrix with coarse features ${\textbf{c}_\mathcal{S}}$ and the regularization parameter on the coarse level is equal to the regularization parameter on the fine level.     
	
	With the coarse level defined, we can create our two-level preconditioner. In AMG the coarse grid correction aims to eliminate the near kernel eigenvectors while the smoothing takes care of the eigenvectors associated with the large eigenvalues. In our approach, the problem is that both smoothing and coarse correction eliminate the error components of the principal eigenvectors. Therefore, we do not apply classical smoothing and recommend to use CG iterations as pre-smoother since CG reduces the norm of the error, but the norm of some of the individual components might increase \cite{doi:10.1137/0730007}. CG is thus sometimes referred to as a rougher instead of a smoother.  
	
	Finally, the coarse size $F_C$ strongly influences the performance of the preconditioner. The spectrum can be divided in three parts: the noise components with small eigenvalues, the principal components represented in the coarse level and lastly the remaining components with eigenvalues in between the noise and the principal components. The coarse size determines the dimension of the principal components and the remaining eigenvectors with nonzero eigenvalues. Depending on the required tolerance on the solution, these remaining components could cause slow convergence and are only eliminated on the fine level. If this happens, there is a possibility that the residual lies in the nullspace of the restriction operator $P^T$ and the principal components are not present in the residual. Therefore, at least one smoothing iteration is required, since only applying the coarse correction would result in a zero solution and cause the fine level Krylov solver to break down.	  

	\subsection{The ideal case}
	In the ideal case, features within one cluster are identical. This means that the element values of the eigenvectors corresponding to the features within one cluster are equal. This places a condition on our approach and the clustering algorithms should collect features in clusters such that the element values, corresponding to the features of one cluster, of the eigenvectors with large eigenvalues do not vary significantly.  
	\begin{definition}[Ideal data set for clustering]\label{def:ideal}
		The ideal data set for clustering \newline $X=[X_1,X_2, ... , X_{F_C}]$ consists of sets of features $X_i\in\mathbb{R}^{N\times n_i}$ consisting of $n_i$ duplicate features $\textbf{x}_i$ for $i=1\dots F_C$.
	\end{definition}
	\begin{theorem}[Ideal clustering case]\label{theorem:ideal}
		Let $X$ satisfy Definition \ref{def:ideal}, the $F_C$ eigenvectors with nonzero eigenvalues of $X^TX$ are constant within the clusters. The prolongation operator $P$ defined in (\ref{eq:prolong}) defines a projection $PP^T$ for these eigenvectors and the eigenvectors of $X^TX$ with nonzero eigenvalues have a one to one correspondence with the eigenvectors of $P^TX^TXP$.
	\end{theorem}
	\begin{proof}
		Given $X=[X_1,X_2, ... , X_{F_C}]$ the ideal data set with $X_i$ consisting of $n_i$ duplicate features $\textbf{x}_i$ for $i=1\dots F_C$. The matrix product $X^TX$ has rank $F_C$ and has the following block structure 
		\begin{equation}
			X^TX=\begin{bmatrix}
			\textbf{x}_1^T\textbf{x}_1\mathbbm{1}^{n_1\times n_1}&	\textbf{x}_1^T\textbf{x}_2\mathbbm{1}^{n_1\times n_2}&\ldots& \textbf{x}_1^T\textbf{x}_{F_C}\mathbbm{1}^{n_1\times n_{F_C}}\\
			\textbf{x}_2^T\textbf{x}_1\mathbbm{1}^{n_2\times n_1}& \ddots & &\vdots\\
			\vdots & &\ddots&\vdots \\
			\textbf{x}_{F_C}^T\textbf{x}_1\mathbbm{1}^{n_{F_C}\times n_1}&\ldots& & \textbf{x}_{F_C}^T\textbf{x}_{F_C}\mathbbm{1}^{n_{F_C}\times n_{F_C}}
			\end{bmatrix}.
		\end{equation}
		The matrix product $X^TX$ can be written as \begin{equation}
		X^TX=PKP^T
		\end{equation} with \begin{equation}
		K=\begin{bmatrix}
		\frac{\textbf{x}_1^T\textbf{x}_1}{n_1}&\frac{\textbf{x}_1^T\textbf{x}_2}{\sqrt{n_1}\sqrt{n_2}} &\ldots&\frac{\textbf{x}_1^T\textbf{x}_{F_C}}{\sqrt{n_1}\sqrt{n_{F_C}}}\\
		\frac{\textbf{x}_2^T\textbf{x}_1}{\sqrt{n_2}\sqrt{n_1}}&
		\ddots&  &\vdots \\
		\vdots& & \ddots&\vdots\\
		\frac{\textbf{x}_{F_C}^T\textbf{x}_1}{\sqrt{n_{F_C}}\sqrt{n_1}}&\ldots& &\frac{\textbf{x}_{F_C}^T\textbf{x}_{F_C}}{n_{F_C}}\\
		\end{bmatrix}.
		\end{equation}
		Using the eigendecomposition of $K=\hat{V}\Lambda \hat{V}^T$, we have \begin{equation}
		X^TX=P\hat{V}\Lambda \hat{V}^T P^T.
		\end{equation}
		By the structure of P, the element values of the eigenvectors with nonzero eigenvalues $P\hat{\textbf{v}}_i=\textbf{v}_i$ are the same within each cluster. Given $P\hat{\textbf{v}}_i=\textbf{v}_i$, we have that $PP^T\textbf{v}_i=P\hat{\textbf{v}}_i=\textbf{v}_i$ and the interpolation operator (\ref{eq:prolong}) defines a projection for the eigenvectors with constant values within the clusters. Using the Rayleigh-Ritz method, the eigenpairs $(\lambda_i,\hat{\textbf{v}}_i)$ of $P^TX^TXP$ can be used to find the original eigenpairs $(\lambda_i,P\hat{\textbf{v}}_i=\textbf{v}_i).$
	\end{proof}	
	\subsection{Clustering}\label{sect:clustering}
	For matrices originating from grids, the coarser levels are defined by aggregating grid elements based on different criteria. Recall that for our applications, the matrix consists of samples with features. Similarity between features can be used to cluster these features. Typically a predefined distance measure $d$ is used to define similarity, widely used distance measures are Euclidean distance, Jaccard distance and the cosine distance. All features within one cluster are then similar to each other and can be represented by the average feature. In this section we look at three clustering algorithms and one sub-sampling technique to create a coarser level. 
	\subsubsection{Leader-follower clustering}
	Leader-follower clustering is given in Algorithm \ref{alg:Leader-Follower} \cite{hartigan1975clustering}. Starting from an initial set of leading clusters, each sample is added to an existing cluster if the sample is similar within the required tolerance to the cluster leader. This cluster leader can then optionally be updated. If a sample is not similar enough to any existing leader, it forms a new cluster.  
	\begin{algorithm}[ht]
		\begin{algorithmic}[1]
			\State \textbf{Initialise:}	given an initial set of clusters $C=\{\textbf{c}_s\}$ and a distance measure $d$
			\For {each sample $\textbf{x}_i$ in $X$}
			\State $\text{index}=\arg \min_{s |\textbf{c}_s\in C} d(\textbf{x}_i,\textbf{c}_s)$ 
			\If {$d(\textbf{x}_i,\textbf{c}_{\text{index}})<\text{Tolerance}$}
			\State Add $\textbf{x}_i$ as follower to the cluster leader $\textbf{c}_{\text{index}}$ and update number of followers $n_{\text{index}}=n_{\text{index}}+1$ 
			\State Optionally update the cluster leader:  $\textbf{c}_{\text{index}}=\textbf{c}_{\text{index}}+\frac{\textbf{x}_i-\textbf{c}_{\text{index}}}{n_{\text{index}}}$
			\Else
			\State Add $\textbf{x}_i$ as cluster leader: $C=C+\{\textbf{x}_i\}$
			\EndIf
			\EndFor
		\end{algorithmic}
		\captionof{algorithm}{Leader-follower clustering}
		\label{alg:Leader-Follower}
	\end{algorithm}
	
	By varying the tolerance more or less clusters are found by Leader-Follower clustering. Small tolerance gives more clusters while high tolerance produces a smaller number of clusters. For our experiments we do not update the leader and preserve sparsity of the leaders if a new sample is added to his followers.
	
	\subsubsection{K-means++}
	Another option to perform clustering on the features is K-means++. The algorithm consists of two phases \cite{arthur2007k}. During the setup phase, K-means++ finds initial $K$ prototypes based on the distance between features. These prototypes can be considered as the leaders for their cluster. The first prototype is chosen at random with the features uniformly distributed. The next initial prototype is randomly chosen with a probability weighted by the distance squared of each data points to the existing prototype. 
	
	The K-means++ algorithm then proceeds to the next phase by assigning each point to a cluster and then updating the prototypes according to its assigned cluster points. Once the prototypes are updated, the data points are reassigned. When the assignments no longer change, the algorithm has converged. 
	\subsubsection{Quadratic R\'enyi-entropy based subsampling} 
	Defining the coarser level can be done by maximizing the entropy of the coarser level or subset. R\'enyi-entropy \cite{Girolami2002Orthogonal,renyi1961measures} based sub-sampling for large-scale data starts from a working set and a training set \cite{de2010optimized}. The entropy of the working set is the decisive parameter and the elements of the working set define the coarse level. Two random data points are chosen from both sets and the R\'enyi-entropy is calculated for both sets. The chosen points from both sets are swapped and the R\'enyi-entropy is recalculated for both sets. If there is an increase in entropy for the working set, the data points are switched else the original sets are maintained. 
	
	The quadratic R\'enyi-entropy for a subset of size $F_C$ is given by $$S_{R2}^{F_C}=-\log\left(\frac{1}{F_C^2|D|^2}\sum_{k=1}^{F_C}\sum_{l=1}^{F_C}\kappa\left(\frac{\textbf{x}_k-\textbf{x}_l}{D\sqrt{2}}\right)\right)$$ where $\kappa$ is a kernel function and $D$ the diagonal of the bandwidth values of the kernel in each dimension. Note that for each pair of data points under investigation the entropy can be updated based on both data points and there is no need to recalculate the entropy completely. These chosen features act then as the prototypes for the clusters and the remaining features are assigned to the cluster of the closest prototype.  

	\subsubsection{Graph Aggregation}
	Clustering of graphs has been used to find communities in networks \cite{schaeffer2007graph,blondel2008fast,lancichinetti2009community,spectralAMG2019Ambra}. A graph $G=(V,E)$ consists of a set of nodes $V$ and edges $E$ connecting two nodes. A symmetric matrix $A$ can be represented by its undirected adjacency graph $G$. For data matrices $X$, computing the product $X^TX$ will result in an symmetric matrix. The adjacency graph for this matrix-product results in a very dense graph, since each feature is highly probable to be connected to a large number other features. Constructing a $k$-nearest neighbor graph, however, results in a sparse graph connecting each feature to its $k$-closest features. Creating a $k$-nearest neighbor graph can be computationally intensive but an approximate graph can be computed efficiently \cite{Muja2014Scalable}. Graph matching has been used for AMG \cite{brannick2012algebraic,brannick2013algebraic,BootCmatchAmbra}. We consider the half-approximate maximum product matching as detailed in BootCMatch\cite{BootCmatchAmbra}. Using the maximum product matching of the graph, pairwise aggregate nodes or clusters are formed as given in Algorithm \ref{alg:graph_aggregation}. 	
	\begin{algorithm}[ht]
		\begin{algorithmic}[1]
			\State \textbf{Given:}	graph $G$ with $F$ nodes and weight matrix $A$
			\State Compute maximum product matching  $M$
			\State Set $F_C=0$ and $U=[1,\dots, F]$
			\While {$U\ne \emptyset}$
			\State Select $i\in U$ at random
			\If {$\exists j\in U\backslash\{i\}$ and $(i,j)\in M$}
			\State $F_C=F_C+1$
			\State Create new aggregate pair $C_{F_C}$: $C_{F_C}=\{i,j\}$
			\State Remove $\{i,j\}$ from $U=U\backslash \{i,j\}$
			\Else
			\State $F_C=F_C+1$
			\State Create new aggregate singleton $C_{F_C}$: $C_{F_C}=\{i\}$
			\State Remove $\{i\}$ from $U=U\backslash \{i\}$			
			\EndIf
			\EndWhile
			\State \textbf{Return:} Aggregates $C_{1},\dots,C_{F_C}$
		\end{algorithmic}
		\captionof{algorithm}{Pairwise aggregation based on maximum product matching\cite{BootCmatchAmbra} }
		\label{alg:graph_aggregation}
	\end{algorithm}
	\subsection{Least-squares interpolation} \label{sect:lsp}
	We assume that the coarse level is able to approximate the principal components. For AMG, it is possible to define the interpolation operator $P$ to exactly interpolate certain prototype vectors in a least-squares manner \cite{brandt2011bootstrap,manteuffel2010operator}. Suppose we have the $K$ first principal components $V=\{\textbf{v}_1,\dots,\textbf{v}_K\}$ and $V\in\mathbb{R}^{F\times K}$. Least squares interpolation minimizes the error squared, made when interpolating these principal components, for each feature $i$ and thus defines each row $P(i,:)$  
	\begin{equation}
	P(i,:)=\arg \min_{p_{ij}, j\in \mathcal{C}_{i}} \sum_{k=1}^K \eta_k\Bigl(\textbf{v}_k(i)-\sum_{j\in \mathcal{C}_{i}}p_{ij}\textbf{v}_k(j)\Bigr)^2 \label{eqn:lsP}
	\end{equation} 
	with $\mathcal{C}_{i}$ the coarse features used to define the value of fine feature $i$ and weights $\eta_k=\langle (X^TX+\beta I)\textbf{v}_k,\textbf{v}_k\rangle $ chosen to reflect the energy norm. For the given aggregation techniques in Section \ref{sect:clustering}, we have that $|\mathcal{C}_{i}|=1$, since each feature is assigned to one cluster or aggregate. This can be extended if multiple clusters are used to interpolate the value of one fine feature. For example, the $k$-closest leaders to feature $i$ of Leader-Follower clustering can be chosen as the coarse features used in the interpolation process. The coarse features $\mathcal{C}_{i}$ for Leader-Follower clustering are the leaders and for R\'enyi-entropy subsampling the features in the working set. For K-means++, there is no natural definition for the coarse feature and one of the features of one cluster is randomly chosen as the representing coarse feature.     
	
	In Bootstrap AMG \cite{brandt2011bootstrap} the prototypes vectors approximate the eigenvectors with small eigenvalues. Bootstrap AMG refines the prototypes used for interpolation by calculating the eigenvectors on the coarsest level. In contrast, we are interested in the principal eigenvectors. In practice, only a few prototype vectors are required and these can be simply calculated using Arnoldi iterations\cite{arnoldi1951principle} on the finest level. 
	
	\subsection{Multilevel preconditioning}\label{sect:multilevel}
	In AMG, coarsening is applied hierarchically until the size is sufficiently small and exact solvers can be applied. The multilevel concept has proven its merits in various applications, see for example \cite{brandt2002multiscale}. The clustering algorithms are used hierarchically to create several levels of coarser features. For K-means and R\'enyi-entropy subsampling the size of the coarse level is chosen. For the leader follower algorithm, the threshold can be varied. The performance of the clusterings algorithms is determined by the correlation of the features. The smaller the dimension of the coarse level, the larger the variance within clusters.
	\begin{algorithm}[ht]
		\caption{Flexible CG \cite{notay2000flexible}}
		\captionof{algorithm}{Flexible CG \cite{notay2000flexible}}
		\label{alg:fpcg}
		\begin{algorithmic}[1]
			\State Given: $m$
			\State Initialize: \begin{itemize}
				\item $\textbf{r}_0=\textbf{b}$
				\item $\textbf{x}_0=\textbf{0}$ 
			\end{itemize} 
			\For{$i=0,\dots, k$}
			\State $\textbf{z}_i=M^{-1}_i\textbf{r}_i$ 
			\State $\textbf{p}_i=\textbf{z}_i-\sum_{k=i-m_i}^{i-1}\frac{\langle\textbf{z}_i,A\textbf{p}_k\rangle}{\langle \textbf{p}_k,A\textbf{p}_k\rangle}\textbf{p}_k$ 
			\State $\textbf{x}_{i+1}=\textbf{x}_i + \frac{\langle\textbf{p}_i,\textbf{r}_i\rangle}{\langle \textbf{p}_i,A\textbf{p}_i\rangle}\textbf{p}_i$
			\State $\textbf{r}_{i+1}=\textbf{r}_i - \frac{\langle\textbf{p}_i,\textbf{r}_i\rangle}{\langle \textbf{p}_i,A\textbf{p}_i\rangle}A\textbf{p}_i$
			\EndFor
		\end{algorithmic}
	Advised truncation strategy: $m_0= 0$; $m_i=\max (1,\mod(i,m+1)) $ and $m>0$, $i>0$.
	\end{algorithm}
	 
 Since we propose the use of a non-stationary smoother, flexible preconditioning is required. Flexible variants for GMRES \cite{saad1993flexible} and CG are available with Flexible CG given in Algorithm \ref{alg:fpcg} \cite{notay2000flexible}. These flexible iterative solvers allow for the coarser levels to be solved inexactly and hence even recursive multilevel Krylov-based cycles are possible \cite{notay2008recursive}. In theory, it is possible to create a multilevel preconditioner. In practice, however, it is more efficient to only use the coarsest level and skip the intermediate levels as experimentally shown in Section \ref{sect:Experiments}. 

\section{Numerical Experiments}\label{sect:Experiments}
We have applied our preconditioner for variants of the linear system  $$(X^TX+\beta I)\textbf{x}=X^T\textbf{b}$$ with $\textbf{b}$ a normal distributed random vector using different feature matrices $X$. {Table \ref{tab:data}} provides the characteristics of the used data matrices $X$ in our experiments. For the larger data sets with  more than $10\ 000$ features, the coarse size $F_C$ was chosen in the range $[3000,7500]$, since calculating the Cholesky decomposition for this dimension range is feasible. Leader-Follower and Graph aggregation do not allow to set the size of coarse level up front. For Leader-Follower, a large threshold is chosen and decreased until the coarse size is in the required range. For large threshold, a Leader-Follower clustering is computed faster than when using a small threshold. Graph aggregation was recursively applied until the coarse level was small enough. For the smaller data sets, the coarse level was chosen such that the number of features was reduced by a factor two or more. The experiments were implemented in C++\footnote{\url{https://scm.cs.kuleuven.be/scm/git/multilevel_macau}} and compiled with gcc 9.3.0 and OPENMP 4.0 with compile option -O3 using a machine with an Intel(R) Core(TM) i7-6560U (2.20GHz) processor with an L3 cache memory of 4096 KB and 16 GB of DRAM where 3 out of 4 cores were used.   

\begin{table}[ht]
	\centering
	\scalebox{0.9}{
		\begin{tabular}{ lp{3cm}lllll} \toprule
			Name&Source&value type &\#rows&\#columns&\#nonzeros&fillin($\%$) \\ \midrule
			Trec10&\cite{davis2011university}&double &106&478&8612&17.00\\
			CNAE&\cite{Lichman2013}&double &1080&856&7233&0.78\\							
			micromass&\cite{Lichman2013}&double &360&1300&48\ 713&10.41\\
			ovariancancer&\cite{MatlabOTB}&double &216&4000&864\ 000&100\\
			gisette&\cite{Lichman2013}&double &6000&5000&29\ 729\ 997&99.1\\
			DrivFace&\cite{Lichman2013}&double &606&6400&3\ 878\ 400&100\\
			air04&\cite{davis2011university}&double &823&8904&72\ 965&1.00\\
			arcene&\cite{Lichman2013}&double &100&10\ 000&540\ 941&54.09\\
			SNP1&\cite{AlphaMPSim}&short int &10\ 000&10\ 000&57\ 641\ 064&57.64\\
			tmc2007& SIAM 2007 Text Mining competition&double &21\ 519&30\ 438&2\ 283\ 179&0.35\\
			nlpdata&\cite{MatlabOTB}&double &31\ 572&34\ 023&2\ 277\ 757&0.21\\				
			rcv1\_multi&\cite{Lewis:2004:RNB:1005332.1005345}&double &15\ 564&47\ 236&1\ 028\ 284&0.14\\
			SNP2&\cite{AlphaMPSim} &short int&10\ 000&50\ 000&282\ 212\ 533&56.44\\
			news20&\cite{keerthi2005modified}&double &15\ 935&62\ 061&1\ 272\ 569&0.13\\
			MSI mouse&\cite{doi:10.1021/ac402540a}&double  &14\ 640&80\ 339&329\ 606\ 165&28.02\\
			SNP3&\cite{AlphaMPSim}&short int &10\ 000&100\ 000&570\ 261\ 477&57.03\\
			E2006&\cite{kogan-etal-2009-predicting}&double &16\ 087&150\ 360&19\ 971\ 015&0.83\\
			ChEMBL&\cite{bento2014chembl}&binary &167\ 668&291\ 714&12\ 246\ 376&0.03	\\					
			\bottomrule
	\end{tabular}}
	\caption{ The characteristics of the data matrices $X$ used in the experiments. }\label{tab:data}
\end{table}

{Table \ref{tab:overview}} shows the execution time for regularized least-squares linear system (LS) \eqref{eq:xtx=xtb} with fixed tolerance $\epsilon$ and $\beta$ for CG, diagonally preconditioned CG (DCG) and flexible CG using a two-level preconditioner (TLFCG) with two iterations of CG as pre-smoother. The coarse level was created using Leader-Follower clustering with Euclidean distance and the adjusted average interpolation in \eqref{eq:prolong}. The Euclidean distance was used for all given experiments. {Table \ref{tab:overview}} additionally provides the number of iterations and the speed-up with respect to CG for DCG and TLFCG. Besides the execution time for solving \eqref{eq:xtx=xtb}, {Table \ref{tab:overview}} details the execution time of Bayesian regression (BR) with noise injection using $600$ samples for CG and TLFCG. For BR the coarse level was solved using the SVD while the linear system uses the Cholesky Decomposition.  
\begin{table}[ht]
	\centering
	\renewcommand{\arraystretch}{1.1}
	\scalebox{0.8}{
		\begin{tabular}{ l|cc|ccc|ccc|cccccc} \toprule
			& &  &\multicolumn{3}{c|}{CG}&\multicolumn{3}{c|}{DCG}&\multicolumn{5}{c}{TLFCG} \\
			Name&$\beta$&$\epsilon$& Iter.& LS(s) &BR(s)& Iter.& LS(s)&SU & Iter.& LS(s)&SU &BR(s)& SU-BR&$F_C$ \\ \midrule
			Trec10&1E-8&1E-8&302 &3.40E-3 &1.131 & 127 &1.86E-3&1.83 & \textbf{1} &\textbf{1.2E-4}&\textbf{28.19} &\textbf{4.17E-1} &\textbf{2.71}&207\\
			CNAE&1E-8&1E-8&785 &1.21E-2 &\textbf{3.43 }& 313 &\textbf{6.15E-03}&\textbf{1.96} & \textbf{107} &2.25E-2&0.54 &4.83&0.71 &375		\\
			micromass&1&1E-6& 1399 &6.86E-2 &1.29E1 &229& 1.02E-2&6.74  &\textbf{1}&\textbf{5.48E-4}&\textbf{125.18} &\textbf{8.21E-1}&\textbf{15.69} &483 \\
			ovariancancer& 1E-8&1E-8&160&1.51E-1&3.41E1 &54&5.85E-2&2.58 &  \textbf{1}& \textbf{5.28E-3}&\textbf{28.6} &\textbf{1.41E1 }& \textbf{2.42}& 435 \\
			gisette&1E-4&1E-4&12\ 603 &4.45E+2&1.81E5 &13\ 119& 4.72E+2&0.94
			 & \textbf{300}&\textbf{5.70E+1}&\textbf{7.81 }&\textbf{1.21E2}&\textbf{1495.87} &2465 \\
			DrivFace&1E-8&1E-8&577&3.53&4.13E2 &604& 2.87& 0.95 &\textbf{1}&\textbf{3.0E-2}&\textbf{ 90.77}&\textbf{9.26E1} &\textbf{4.46} &2028 \\			
			air04& 1E-8&1E-8&440&3.58E-2&\textbf{1.68E1 }&400&4.35E-2&0.82 &\textbf{2}&\textbf{2.44E-2}& \textbf{1.47}&8.57E1 & 0.20& 3045 \\
			arcene&1E-4&1E-8&55&2.96E-2&8.85 &51&3.27E-2&0.91 &\textbf{1}&	\textbf{1.73E-2}&\textbf{1.71}&\textbf{6.36 }&\textbf{1.39} & 3274 \\
			SNP1&1E-2 &1E-3 	&11061 &6.89E2 & 3.64E3 &14393 &8.98E2 & 0.76      &\textbf{97} &\textbf{3.19E1} &\textbf{21.60} &\textbf{8.56E2}&\textbf{4.25} &3216 \\
			tmc2007&1E-4 &1E-4&3875&1.93E+1&4.17E3 &506&\textbf{2.61 }&\textbf{7.41}  &\textbf{167}&8.76&	2.20&\textbf{1.92E3} &\textbf{2.17} &4728\\
			nlpdata& 1E-4&1E-3& 98\ 093&3.31E+2&1.56E5 &8805&3.32E+1&9.97 &\textbf{227}&\textbf{1.40E+1}&\textbf{23.60}&\textbf{6.75E2} &\textbf{231.11} &5793	\\
			rcv1\_multi&1E-4&1E-3&1523 &2.42 &\textbf{1.43E3} & 1030 &\textbf{1.96} &\textbf{1.24}  &\textbf{201}&5.96&0.41 &3.55E3&0.40 &3944  \\
			SNP2&1E-2 &1E-3 	&18\ 371 &6.11E+3 &3.91E4 	  &24\ 581 &7.14E+3 &0.86       &\textbf{53} &\textbf{8.58E+1 } & \textbf{71.19}& \textbf{6.69E3}& \textbf{5.85}&5582 \\
			news20&1E-6 &1E-6 	&21\ 830 &6.59E+1 & \textbf{1.288E4}	  &21\ 268 &7.09E+1 &0.93      &\textbf{707} &\textbf{4.88E+1 }&\textbf{1.35 }&1.38E4&0.93 &6151 \\			
			MSI mouse&1E-6 &1E-6 	&675 &3.44E2 & 1.12E6	  &108 &\textbf{5.65E1} & \textbf{6.09}      &\textbf{30} &8.00E1 &4.31 &\textbf{5.84E5} &\textbf{1.91} &5269 \\
			SNP3&1E-2 &1E-3 	&11\ 512 &7.45E+3 &6.84E4 	  &10\ 358 &6.45E+3 & 1.16     &\textbf{23} &\textbf{1.21E+2} &\textbf{61.61} &\textbf{1.52E4} & \textbf{4.5} &4893 \\	
			E2006&1E-6 &1E-6 	&3943 &1.74E+2 & 2.52E2	  &5529 &2.48E2 &0.70       &\textbf{203} &\textbf{5.97E1} &\textbf{2.91} &\textbf{2.38E2} & \textbf{1.06}&7036 \\
			ChEMBL&1E-3&1E-3   &11790 &2.43E2 &5.79E4   &11535 &2.66E2 &0.91    &\textbf{712} &\textbf{1.06E2}&\textbf{2.29} &\textbf{4.14E4} &\textbf{1.40} &5581 \\ 							
			\bottomrule
	\end{tabular}}
	\caption{ The number of iterations and execution time for solving regularized least squares linear systems using CG, diagonally preconditioned CG (DCG) and FCG accelerated with our two-level preconditioner (TLFCG) for different data sets. Additionally the execution time for Bayesian regression using 600 samples is reported for CG and TLFCG. The reported speed-up (SU) is with respect for CG for solving the linear system. For TLFCG, the speed-up for Bayesian regression (SU-BR) is also given. TLFCG achieves speed-up for most of the data sets and solves the linear system with the smallest number of iterations.}\label{tab:overview}
\end{table}

Most of the data sets in {Table \ref{tab:data}} are underdetermined or undersampled. Regularization is added and by setting $\beta>0$, noise components with large norms are penalized. In machine learning, the regularization parameter is often used as a measure against overfitting and several values are tried using cross-validation. Furthermore, by increasing $\beta$ the condition number decreases and CG converges faster. The values of $\beta$ and $\epsilon$ in {Table \ref{tab:overview}} are chosen from a set for which CG converges slow but within a reasonable number of iterations. More results for varying $\beta$ and $\epsilon$ are given in Section \ref{sect:regu}. Note that if $\beta$ is chosen very large, CG will converge fast and no preconditioner is required. For BR, the value of $\beta$ is determined automatically within the sampling process and the reported tolerance $\epsilon$ is used. 

As can be seen from {Table \ref{tab:overview}}, applying our two-level preconditioner results in speed-up with respect to CG for most of the data sets. Diagonally preconditioned CG can reduce the number of iterations but the gain is often small or there is no gain at all. TLFCG requires less iterations to reach the required tolerance for all cases but each iteration is more expensive. The actual speed-up depends on the quality of clusters. The goal is to achieve clusters for which the correlation within a cluster a high but correlation between clusters is small which is data dependent. Note that the gisette data set is nearly dense with more samples than features and other techniques to solve the linear system are better suited than CG. 

\begin{table}[ht]
	\centering
	\renewcommand{\arraystretch}{1.1}
	\scalebox{0.8}{
		\begin{tabular}{ l|ccc|cc|cc|cc} \toprule
			&\multicolumn{3}{c|}{LF}&\multicolumn{2}{c|}{KM}&\multicolumn{2}{c|}{RE}&\multicolumn{2}{c}{GA} \\
			Name&$F_C$&Cholesky(s)&SVD(s) &$F_C$&Cholesky(s)&$F_C$&Cholesky(s)& $F_C$&Cholesky(s)\\ \midrule
			Trec10&207&1.532E-2&2.57E-2&200&1.71E-2&200&1.4E-2&262&4.71E-3 \\
			CNAE &375&1.52E-2&6.32E-2&	400& 2.75E-2&400&1.51E-1&528 &5.21E-2\\
			micromass &483& 9.01E-2& 1.01E-1& 500&1.97E-1&500&4.01E-1&745&9.09E-2\\
			ovariancancer& 435&1.08&8.34E-1&400&8.29&400&1.42& 510&1.62\\
			gisette&2465&8.58E1 &1.23E2&3000&1.01E3&3000&2.72E2& 3583&8.4E1\\
			DrivFace&2028& 8.17&4.31 &2000&9.00E1&2000&2.17E1&3441&1.11E1 \\			
			air04 & 3045&4.26E-1&3.11&3000&4.76&3000&2.28&5069&1.14 \\
			arcene& 3274&1.90& 1.78& 3000&2.67E1&3000&8.05&5572&2.8\\
			SNP1 &3216 &1.56E2 &2.02E2&5000&1.89E3&5000&2.1E3&5564&2.73E3\\
			tmc2007&4728&6.03& 1.58E3 &5000&2.77E2&5000&3.87E1 \\
			nlpdata &5793&7.11&1.58E3 &6000&3.02E2&6000&2.99E1&5413&2.20E1	\\
			rcv1\_multi&3944 &2.73 &5.79E2&4000&1.79E2& 4000&7.63\\
			SNP2&5582 &1.04E3 &1.11E3&5000& 7.37E3&5000&3.86E3&4720&5.31E3\\
			news20&6151 &6.09 &2.05E2&6000&3.79E2& 6000 &2.14E1 & 7059&7.06E1\\			
			E2006&7036 &5.59E1 &2.60E3&7000 &8.80E3  &7000&3.34E2 & 6955&4.01E2\\
			ChEMBL& 5581 &2.28E1&1.58E3 & & & & \\
			\bottomrule
	\end{tabular}}
	\caption{ The setup time of our two-level preconditioner. The reported times are the sum of the clustering time and the time to create the coarse solver. The singular value decomposition (SVD) is only reported in combination with Leader-follower (LF) clustering. For K-means++ (KM), R\'enyi-entropy subsampling and maximum product graph aggregation (GA), the setup time in combination with the Cholesky decomposition is reported for most data sets. The setup time is related to the number of nonzeros of $X$ as given in Table \ref{tab:data}.   }\label{tab:setup_time}
\end{table}

The setup time for TLFCG is detailed in {Table \ref{tab:setup_time}} for Leader-Follower clustering (LF), K-means++ (KM), R\'enyi-entropy based subsampling (RE) and graph aggregation  (GA) using a $5$-nearest neighbour graph. These reported times are in combination with a Cholesky decomposition on the coarsest level. For LF, the setup time using a singular value decomposition (SVD) is additionally given. Computing the Cholesky decomposition or SVD was done using LAPACK\cite{laug}. For SVD, only the right singular vectors of $X_C$ are required to solve equation \eqref{eq:xtx=xtb} on the coarse level. As can be seen from {Table \ref{tab:setup_time}} the setup time is directly related to the number of nonzeros (see {Table \ref{tab:data}}).  

The sparser the data, the smaller the clustering time. The difference in setup time between Cholesky and SVD depends on the number of rows of $X_C$. If the number of rows is higher than the number of columns, calculating the SVD will result in larger execution time. The optimal value of $\beta$ on the coarse level is the same as the value of $\beta$ on the finest level. This means that when using the same coarse matrix $X_C$, the Cholesky decomposition has to be recalculated if $\beta$ varies while the SVD can be reused. Note that the execution time of computing the Cholesky composition can be insignificant with respect to the time solving the linear system. As example for the ChEMBL data, computing only the Cholesky decomposition for a coarse level with $5581$ features takes about one second while solving the linear system takes about four minutes.  

\subsection{Designing the two-level preconditioner}\label{sec:interpclust}
There are different ways to create the two-level preconditioner. Firstly, there are different clustering algorithms. We have detailed Leader-Follower (LF), K-means++ (KM), R\'{e}nyi-entropy based subsampling (RE) and graph based aggregation (GA). These are basic but generally applicable clustering algorithms. Next, there are different ways to design the interpolation operator $P$ and subsequently create the coarse level from a clustering. We defined using the plain average (A), the adjusted squareroot average (SA) as detailed in \eqref{eq:prolong} and the least-squares interpolation used in bootstrap (LS). Figure \ref{fig:clustering} shows the relative residual in function of the numbers of iterations for different clustering algorithms and different interpolation operators for four data sets using the same values for $\beta$ and $\epsilon$ as given in Table~\ref{tab:overview}. The dimension of the coarse levels are given in Table \ref{tab:setup_time}. The coarse sizes for KM and RE  were chosen close to the values of LF and GA. The coarse level in these experiments was solved using a Cholesky Decomposition. 
\begin{figure}
	\centering
	\begin{subfigure}{.495\textwidth}
		\tikzsetnextfilename{news20ca}%
		\resizebox{\linewidth}{!}{\begin{tikzpicture}
\begin{axis}[
height=8cm,
width=12cm,
cycle list name=clusters,
ymode = log,
grid=both,
grid style={line width=.1pt, draw=gray!20},
major grid style={line width=.2pt,draw=gray!60},
minor tick num=5,
xmin=0,xmax=1500,
xlabel=Iteration,ylabel=$\frac{||X^T\textbf{b}-(X^TX+\beta I)\textbf{x}||}{||X^T\textbf{b}||}$,
legend columns=2, 
legend pos=north east,
legend style={font=\small}
]
\addplot+[mark=none] table[x=iter, y=res] {Tikz/news20_lf_pffff_1e-61e-6.dat};
\addplot+[mark=none] table[x=iter, y=res] {Tikz/news20_graph_pffff_1e-61e-6.dat};
\addplot+[mark=none] table[x=iter, y=res] {Tikz/news20_km_pffff_1e-61e-6.dat};
\addplot+[mark=none] table[x=iter, y=res] {Tikz/news20_re_pffff_1e-61e-6.dat};
\legend{LF-SA,GA-SA,KM-SA,RE-SA}
\end{axis}
\end{tikzpicture}   }
		\caption{news20}
		\label{fig:news20_clusteringa}
	\end{subfigure}
	\begin{subfigure}{.495\textwidth}
	\tikzsetnextfilename{news20cb}%
	\resizebox{\linewidth}{!}{\begin{tikzpicture}
\begin{axis}[
height=8cm,
width=12cm,
cycle list name=clusters,
ymode = log,
grid=both,
grid style={line width=.1pt, draw=gray!20},
major grid style={line width=.2pt,draw=gray!60},
minor tick num=5,
xmin=0,xmax=750,
xlabel=Iteration,ylabel=$\frac{||X^T\textbf{b}-(X^TX+\beta I)\textbf{x}||}{||X^T\textbf{b}||}$,
legend columns=1,
legend pos=north east,
legend style={font=\small}
]
\addplot+[mark=none] table[x=iter, y=res] {Tikz/news20_lf_pffff_1e-61e-6.dat};
\addplot+[mark=none] table[x=iter, y=res] {Tikz/news20_lf_pttff_1e-61e-6.dat};
\addplot+[mark=none] table[x=iter, y=res] {Tikz/news20_lf_pfftt_1e-61e-6.dat};
\legend{LF-SA,LF-LS,LF-A}
\end{axis}
\end{tikzpicture}   }
	\caption{news20}
	\label{fig:news20_clusteringb}
\end{subfigure}
	\begin{subfigure}{.495\textwidth}
	\tikzsetnextfilename{E2006ca}%
	\resizebox{\linewidth}{!}{\begin{tikzpicture}
\begin{axis}[
height=8cm,
width=12cm,
cycle list name=clusters,
ymode = log,
grid=both,
grid style={line width=.1pt, draw=gray!20},
major grid style={line width=.2pt,draw=gray!60},
minor tick num=5,
xmin=0,xmax=350,
xlabel=Iteration,ylabel=$\frac{||X^T\textbf{b}-(X^TX+\beta I)\textbf{x}||}{||X^T\textbf{b}||}$,
legend columns=2,
legend pos=north east,
legend style={font=\small}
]
\addplot+[mark=none] table[x=iter, y=res] {Tikz/E2006_lf_pffff_1e-61e-6.dat};
\addplot+[mark=none] table[x=iter, y=res] {Tikz/E2006_graph_pffff_1e-61e-6.dat};
\addplot+[mark=none] table[x=iter, y=res] {Tikz/E2006_km_pffff_1e-61e-6.dat};
\addplot+[mark=none] table[x=iter, y=res] {Tikz/E2006_re_pffff_1e-61e-6.dat};
\legend{LF-SA,GA-SA,KM-SA,RE-SA}
\end{axis}
\end{tikzpicture}   }
	\caption{E2006}
	\label{fig:E2006_clusteringa}
	\end{subfigure}
	\begin{subfigure}{.495\textwidth}
	\tikzsetnextfilename{E2006cb}%
	\resizebox{\linewidth}{!}{\begin{tikzpicture}
\begin{axis}[
height=8cm,
width=12cm,
cycle list name=clusters,
ymode = log,
grid=both,
grid style={line width=.1pt, draw=gray!20},
major grid style={line width=.2pt,draw=gray!60},
minor tick num=5,
xmin=0,xmax=350,
xlabel=Iteration,ylabel=$\frac{||X^T\textbf{b}-(X^TX+\beta I)\textbf{x}||}{||X^T\textbf{b}||}$,
legend columns=1,
legend pos=north east,
legend style={font=\small}
]
\addplot+[mark=none] table[x=iter, y=res] {Tikz/E2006_lf_pffff_1e-61e-6.dat};
\addplot+[mark=none] table[x=iter, y=res] {Tikz/E2006_lf_pttff_1e-61e-6.dat};
\addplot+[mark=none] table[x=iter, y=res] {Tikz/E2006_lf_pfftt_1e-61e-6.dat};
\legend{LF-SA,LF-LS,LF-A}
\end{axis}
\end{tikzpicture}   }
	\caption{E2006}
	\label{fig:E2006_clusteringb}
\end{subfigure}
	\begin{subfigure}{.495\textwidth}
	\tikzsetnextfilename{SNP2ca}%
	\resizebox{\linewidth}{!}{\begin{tikzpicture}
\begin{axis}[
height=8cm,
width=12cm,
cycle list name=clusters,
ymode = log,
grid=both,
grid style={line width=.1pt, draw=gray!20},
major grid style={line width=.2pt,draw=gray!60},
minor tick num=5,
xmin=0,xmax=75,
xlabel=Iteration,ylabel=$\frac{||X^T\textbf{b}-(X^TX+\beta I)\textbf{x}||}{||X^T\textbf{b}||}$,
legend columns=2,
legend pos=north east,
legend style={font=\small}
]
\addplot+[mark=none] table[x=iter, y=res] {Tikz/SNP2_lf_pffff_1e-61e-6.dat};
\addplot+[mark=none] table[x=iter, y=res] {Tikz/SNP2_graph_pffff_1e-61e-6.dat};
\addplot+[mark=none] table[x=iter, y=res] {Tikz/SNP2_km_pffff_1e-61e-6.dat};
\addplot+[mark=none] table[x=iter, y=res] {Tikz/SNP2_re_pffff_1e-61e-6.dat};
\legend{LF-SA,GA-SA,KM-SA,RE-SA}
\end{axis}
\end{tikzpicture}   }
	\caption{SNP2}
	\label{fig:SNP2_clusteringa}
\end{subfigure}
	\begin{subfigure}{.495\textwidth}
	\tikzsetnextfilename{SNP2cb}%
	\resizebox{\linewidth}{!}{\begin{tikzpicture}
\begin{axis}[
height=8cm,
width=12cm,
cycle list name=clusters,
ymode = log,
grid=both,
grid style={line width=.1pt, draw=gray!20},
major grid style={line width=.2pt,draw=gray!60},
minor tick num=5,
xmin=0,xmax=100,
xlabel=Iteration,ylabel=$\frac{||X^T\textbf{b}-(X^TX+\beta I)\textbf{x}||}{||X^T\textbf{b}||}$,
legend columns=1,
legend pos=north east,
legend style={font=\small}
]
\addplot+[mark=none] table[x=iter, y=res] {Tikz/SNP2_lf_pffff_1e-61e-6.dat};
\addplot+[mark=none] table[x=iter, y=res] {Tikz/SNP2_lf_pttff_1e-61e-6.dat};
\addplot+[mark=none] table[x=iter, y=res] {Tikz/SNP2_lf_pfftt_1e-61e-6.dat};
\legend{LF-SA,LF-LS,LF-A}
\end{axis}
\end{tikzpicture}   }
	\caption{SNP2}
	\label{fig:SNP2_clusteringb}
\end{subfigure}
	\begin{subfigure}{.495\textwidth}
		\tikzsetnextfilename{nlpdataca}%
	\resizebox{\linewidth}{!}{\begin{tikzpicture}
\begin{axis}[
height=8cm,
width=12cm,
cycle list name=clusters,
ymode = log,
grid=both,
grid style={line width=.1pt, draw=gray!20},
major grid style={line width=.2pt,draw=gray!60},
minor tick num=5,
xmin=0,xmax=300,
xlabel=Iteration,ylabel=$\frac{||X^T\textbf{b}-(X^TX+\beta I)\textbf{x}||}{||X^T\textbf{b}||}$,
legend columns=2,
legend pos=north east,
legend style={font=\small}
]
\addplot+[mark=none] table[x=iter, y=res] {Tikz/nlpdata_lf_pffff_1e-61e-6.dat};
\addplot+[mark=none] table[x=iter, y=res] {Tikz/nlpdata_graph_pffff_1e-61e-6.dat};
\addplot+[mark=none] table[x=iter, y=res] {Tikz/nlpdata_km_pffff_1e-61e-6.dat};
\addplot+[mark=none] table[x=iter, y=res] {Tikz/nlpdata_re_pffff_1e-61e-6.dat};
\legend{LF-SA,GA-SA,KM-SA,RE-SA}
\end{axis}
\end{tikzpicture}   }
	\caption{nlpdata}
	\label{fig:nlp_clusteringa}
\end{subfigure}
	\begin{subfigure}{.495\textwidth}
	\tikzsetnextfilename{nlpdatacb}%
	\resizebox{\linewidth}{!}{\begin{tikzpicture}
\begin{axis}[
height=8cm,
width=12cm,
cycle list name=clusters,
ymode = log,
grid=both,
grid style={line width=.1pt, draw=gray!20},
major grid style={line width=.2pt,draw=gray!60},
minor tick num=5,
xmin=0,xmax=260,
xlabel=Iteration,ylabel=$\frac{||X^T\textbf{b}-(X^TX+\beta I)\textbf{x}||}{||X^T\textbf{b}||}$,
legend columns=1,
legend pos=north east,
legend style={font=\small}
]
\addplot+[mark=none] table[x=iter, y=res] {Tikz/nlpdata_lf_pffff_1e-61e-6.dat};
\addplot+[mark=none] table[x=iter, y=res] {Tikz/nlpdata_lf_pttff_1e-61e-6.dat};
\addplot+[mark=none] table[x=iter, y=res] {Tikz/nlpdata_lf_pfftt_1e-61e-6.dat};
\legend{LF-SA,LF-LS,LF-A}
\end{axis}
\end{tikzpicture}   }
	\caption{nlpdata}
	\label{fig:nlp_clusteringb}
\end{subfigure}
	\caption{The relative residual in function of the number of iterations for different clustering algorithms (left) and different interpolation operators (right) for news20, E2006, SNP2 and nlpdata. The results left are depict the performance of Leader-follower clustering (LF), K-means++ (KM), R\'enyi-entropy subsampling and graph aggregation (GA). The subfigures on the right show the different interpolation operators: adjusted squareroot average (SA), average (A) and least squares interpolation (LS). LF in combination with SA or LS interpolation works well.}
	\label{fig:clustering}
\end{figure}

As can be seen in the left column of Figure \ref{fig:clustering}, the efficiency of the aggregation technique is data dependent. Leader-follower performs well for these four data sets. Note that using KM or LF results in significant reduction of number of iterations with respect to CG. For R\'enyi-entropy subsampling, the radial basis function (RBF) kernel $\mathcal{k}(x,y)= \exp\left(- \frac{||x-y ||^2}{2\sigma^2}\right)$ with $\sigma=0.6$ was used. The choice of kernel can be further optimized depending on the data. Since a kernel is required, RE is less optimal as a black-box solver. Graph aggregation using the maximum product matching performs well for the SNP2 data set, but less optimal for the other three data sets.    

The right column of Figure \ref{fig:clustering} details the different interpolations methods. The least-squares interpolation and the adjusted average \eqref{eq:prolong} have similar performance for these data sets with the least-squares interpolation slightly better for 3 data sets. Note that both these interpolation operators reduce the number of iterations significantly. Taking the average performs badly for two data sets, as seen in Figures \ref{fig:SNP2_clusteringb} and \ref{fig:E2006_clusteringb}.     
\begin{figure}
	\centering
	\begin{subfigure}{.495\textwidth}
	\tikzsetnextfilename{news20sa}%
	\resizebox{\linewidth}{!}{\begin{tikzpicture}
\begin{axis}[
height=8cm,
width=12cm,
cycle list name=smoothers,
ymode = log,
grid=both,
grid style={line width=.1pt, draw=gray!20},
major grid style={line width=.2pt,draw=gray!60},
minor tick num=5,
xmin=0,xmax=200,
xlabel=Iteration,ylabel=$||X^T\textbf{b}-(X^TX+\beta I)\textbf{x}||$,
legend columns=1,
legend pos=north east,
legend style={font=\small}
]
\addplot+[mark=none] table[x=iacg, y=racg] {Tikz/news20_smoothing2.dat};
\addplot+[mark=none] table[x=ippcg, y=rppcg] {Tikz/news20_smoothing2.dat};
\addplot+[mark=none] table[x=ipcg, y=rpcg] {Tikz/news20_smoothing2.dat};
\addplot+[mark=none] table[x=ipgm, y=rpgm] {Tikz/news20_smoothing2.dat};
\legend{Pre-CG,PrePost-CG,Post-CG,Post-GMRES}
\end{axis}
\end{tikzpicture}   }
		\caption{news20}
	\label{fig:news20_smoothinga}
	\end{subfigure}
	\begin{subfigure}{.495\textwidth}
	\tikzsetnextfilename{news20sb}%
	\resizebox{\linewidth}{!}{\begin{tikzpicture}
\begin{axis}[
height=8cm,
width=12cm,
cycle list name=smoothers,
ymode = log,
grid=both,
grid style={line width=.1pt, draw=gray!20},
major grid style={line width=.2pt,draw=gray!60},
minor tick num=5,
xmin=0,xmax=200,
xlabel=Iteration,ylabel=$||X^T\textbf{b}-(X^TX+\beta I)\textbf{x}||$,
legend columns=1,
legend pos=north east,
legend style={font=\small}
]
\addplot+[mark=none] table[x=iacg, y=racg] {Tikz/news20_smoothing2.dat};
\addplot+[mark=none] table[x=iagm, y=ragm] {Tikz/news20_smoothing2.dat};
\addplot+[mark=none] table[x=iars, y=rars] {Tikz/news20_smoothing2.dat};
\legend{Pre-CG,Pre-GMRES, Pre-RS}
\end{axis}
\end{tikzpicture}   }
	\caption{news20}
	\label{fig:news20_smoothingb}
\end{subfigure}
	\begin{subfigure}{.495\textwidth}
	\tikzsetnextfilename{E2006sa}%
	\resizebox{\linewidth}{!}{\begin{tikzpicture}
\begin{axis}[
height=8cm,
width=12cm,
cycle list name=smoothers,
ymode = log,
grid=both,
grid style={line width=.1pt, draw=gray!20},
major grid style={line width=.2pt,draw=gray!60},
minor tick num=5,
xmin=0,xmax=100,
xlabel=Iteration,ylabel=$||X^T\textbf{b}-(X^TX+\beta I)\textbf{x}||$,
legend columns=1,
legend pos=north east,
legend style={font=\small}
]
\addplot+[mark=none] table[x=iacg, y=racg] {Tikz/E2006_smoothing2.dat};
\addplot+[mark=none] table[x=ippcg, y=rppcg] {Tikz/E2006_smoothing2.dat};
\addplot+[mark=none] table[x=ipcg, y=rpcg] {Tikz/E2006_smoothing2.dat};
\addplot+[mark=none] table[x=ipgm, y=rpgm] {Tikz/E2006_smoothing2.dat};
\legend{Pre-CG,PrePost-CG,Post-CG,Post-GMRES}
\end{axis}
\end{tikzpicture}   }
	\caption{E2006}
	\label{fig:E2006_smoothinga}
\end{subfigure} 
	\begin{subfigure}{.495\textwidth}
	\tikzsetnextfilename{E2006sb}%
	\resizebox{\linewidth}{!}{\begin{tikzpicture}
\begin{axis}[
height=8cm,
width=12cm,
cycle list name=smoothers,
ymode = log,
grid=both,
grid style={line width=.1pt, draw=gray!20},
major grid style={line width=.2pt,draw=gray!60},
minor tick num=5,
xmin=0,xmax=100,
xlabel=Iteration,ylabel=$||X^T\textbf{b}-(X^TX+\beta I)\textbf{x}||$,
legend columns=1,
legend pos=north east,
legend style={font=\small}
]
\addplot+[mark=none] table[x=iacg, y=racg] {Tikz/E2006_smoothing2.dat};
\addplot+[mark=none] table[x=iagm, y=ragm] {Tikz/E2006_smoothing2.dat};
\addplot+[mark=none] table[x=iars, y=rars] {Tikz/E2006_smoothing2.dat};
\legend{Pre-CG,Pre-GMRES, Pre-RS}
\end{axis}
\end{tikzpicture}   }
	\caption{E2006}
	\label{fig:E2006_smoothingb}
\end{subfigure}  
	\caption{The residual in function of the number of iterations for different smoothing options for news20 and E2006. The used smoothers are CG, GMRES and Richardson-iteration (RS). The smoothers were applied before the coarse correction (Pre) , after the coarse correction (Post) and before and after the coarse correction (PrePost).}
	\label{fig:smoothing}
\end{figure}

After designing the coarse level, we propose to use CG as rougher. Figure \ref{fig:smoothing} shows the absolute residual in function the iteration for applying two iterations of CG before the coarse-correction (Pre-CG), before and after the coarse correction (PrePost-CG) or after the coarse correction (Post-CG). Figure \ref{fig:smoothing} additionally shows the result when applying GMRES or Richardson iteration instead of CG. As can be seen from Figures \ref{fig:news20_smoothinga} and \ref{fig:E2006_smoothinga}, applying CG before the coarse correction results in the faster and more stable convergence. Using CG as a rougher additionally has a small advantage over GMRES and Richardson iteration as shown in Figures \ref{fig:news20_smoothingb} and \ref{fig:E2006_smoothingb}. 


The chosen aggregation method results in a coarser level. This aggregation method can be recursively applied resulting in a hierarchy of levels. In theory a multilevel preconditioner can be created. Table \ref{tab:ml} however shows that in practice it is more efficient to only use the coarsest level. The number of fine level iterations is reduced by using the intermediate levels, but the overall execution time is larger since more time is spent on the intermediate levels. The two-level preconditioner can be seen as a subspace preconditioner combined with CG as a rougher. 
\begin{table}
	\begin{subtable}[t]{.5\linewidth}%
		\centering
		\renewcommand{\arraystretch}{1.1}
		\scalebox{0.75}{
			\begin{tabular}{c|cc|cc|cc}\toprule
				nb&\multicolumn{2}{c|}{$\epsilon_{\text{sub}}=1e-1$}&\multicolumn{2}{c|}{$\epsilon_{\text{sub}}=1e-3$}&\multicolumn{2}{c}{Two-level}\\
				levels& Iter. & Exec. (s)& Iter. & Exec. (s)& Iter. & Exec. (s)  \\\midrule
				3&67 &11.48 &13 &137.69 &55 &4.44 \\
			4&19 &48.57 &2 &77.11 &54 &  4.19\\
				\bottomrule
		\end{tabular}}
		\caption{news20}\label{tab:mlnews20}
	\end{subtable}%
	\begin{subtable}[t]{.5\linewidth}%
		\centering
		\renewcommand{\arraystretch}{1.1}
		\scalebox{0.75}{
	\begin{tabular}{c|cc|cc|cc}\toprule
		nb&\multicolumn{2}{c|}{ $\epsilon_{\text{sub}}=1e-1$}&\multicolumn{2}{c|}{$\epsilon_{\text{sub}}=1e-3$}&\multicolumn{2}{c}{Two-level}\\
		levels& Iter. & Exec. (s)& Iter. & Exec. (s)& Iter. & Exec. (s) \\\midrule
		3&2 & 1.09 &1 &7.32 &3 &0.83\\
		4&2 &1.28&1 &61.23 &2 &1.06  \\
		5&1 &45.21&1 &365.64 &3 &0.82  \\
		\bottomrule
	\end{tabular}}
		\caption{E2006}\label{tab:mlE2006}
	\end{subtable}%
	\caption{The number of fine level iterations and executions time for recursively applying FCG with different tolerances $\epsilon_{\text{sub}}$ on each level and directly projecting to the coarsest level (Two-level) for news20 and E2006. The required tolerance $\epsilon$ was set to $1E-3$ and $\beta=1E-6$. The size of the different levels for news20 are $[62061, 27543, 16315, 7101]$ and for E2006 $[150360,42337,27865,20629,8463$]. The coarsest levels are used in combination with the finest level if not all available levels are exploited in the preconditioner. }\label{tab:ml}
\end{table}

\subsection{Influence of the regularization parameter} \label{sect:regu}
The regularization parameter $\beta$ plays an important role in the performance of an iterative solver. Figure \ref{fig:betatol} shows the speed-up in execution time of two-level preconditioned FCG with respect to CG for different values of $\beta$ and different values of $\epsilon$. There are two phenomenons visible in the figure. Firstly, the larger the value of the regularization parameter $\beta$, the faster CG will converge and generally the smaller the speed-up we can achieve. The actual convergence depends on the eigenspectrum and is thus data dependent. Note that the eigenspectrum is influenced by the value of $\beta$. 

\begin{figure}[h]
	\centering
	\begin{subfigure}{.45\textwidth}
		\tikzsetnextfilename{news20betatol}%
		\resizebox{\linewidth}{!}{\begin{tikzpicture}
\begin{axis}[
height=8cm,
width=12cm,
xmode = log,
cycle list name=clusters,
grid=both,
grid style={line width=.1pt, draw=gray!20},
major grid style={line width=.2pt,draw=gray!60},
minor tick num=5,
xlabel=$\epsilon$,ylabel=Speed-up,
legend columns=1,
legend pos=north west,
legend style={font=\small}
]
\addplot+ table[x=tol, y=b1] {Tikz/news20_beta_tol.dat};
\addplot+ table[x=tol, y=b2] {Tikz/news20_beta_tol.dat};
\addplot+ table[x=tol, y=b3] {Tikz/news20_beta_tol.dat};
\addplot+ table[x=tol, y=b4] {Tikz/news20_beta_tol.dat};
\addplot+ table[x=tol, y=b5] {Tikz/news20_beta_tol.dat};
\legend{$\beta=1e-6$,$\beta=1e-5$,$\beta=1e-4$,$\beta=1e-3$,$\beta=1e-2$}
\end{axis}
\end{tikzpicture}   }
		\caption{news20}
		\label{fig:news20_betatol}
	\end{subfigure}
	\begin{subfigure}{.45\textwidth}
		\tikzsetnextfilename{E2006betatol}%
		\resizebox{\linewidth}{!}{\begin{tikzpicture}
\begin{axis}[
height=8cm,
width=12cm,
xmode = log,
cycle list name=clusters,
grid=both,
grid style={line width=.1pt, draw=gray!20},
major grid style={line width=.2pt,draw=gray!60},
minor tick num=5,
xlabel=$\epsilon$,ylabel=Speed-up,
legend columns=1,
legend pos=north west,
legend style={font=\small}
]
\addplot+ table[x=tol, y=b1] {Tikz/E2006_beta_tol.dat};
\addplot+ table[x=tol, y=b2] {Tikz/E2006_beta_tol.dat};
\addplot+ table[x=tol, y=b3] {Tikz/E2006_beta_tol.dat};
\addplot+ table[x=tol, y=b4] {Tikz/E2006_beta_tol.dat};
\addplot+ table[x=tol, y=b5] {Tikz/E2006_beta_tol.dat};
\legend{$\beta=1e-6$,$\beta=1e-5$,$\beta=1e-4$,$\beta=1e-3$,$\beta=1e-2$}
\end{axis}
\end{tikzpicture}   }
		\caption{E2006}
		\label{fig:E2006_betatol}
	\end{subfigure}
	\begin{subfigure}{.45\textwidth}
		\tikzsetnextfilename{SNP2betatol}%
		\resizebox{\linewidth}{!}{\begin{tikzpicture}
\begin{axis}[
height=8cm,
width=12cm,
xmode = log,
cycle list name=clusters,
grid=both,
grid style={line width=.1pt, draw=gray!20},
major grid style={line width=.2pt,draw=gray!60},
minor tick num=5,
xlabel=$\epsilon$,ylabel=Speed-up,
legend columns=1,
legend pos=north west,
legend style={font=\small}
]
\addplot+ table[x=tol, y=b1] {Tikz/SNP2_beta_tol.dat};
\addplot+ table[x=tol, y=b2] {Tikz/SNP2_beta_tol.dat};
\addplot+ table[x=tol, y=b3] {Tikz/SNP2_beta_tol.dat};
\addplot+ table[x=tol, y=b4] {Tikz/SNP2_beta_tol.dat};
\legend{$\beta=1e-2$,$\beta=1e-1$,$\beta=1$,$\beta=10$}
\end{axis}
\end{tikzpicture}   }
		\caption{SNP2}
		\label{fig:SNP2_betatol}
	\end{subfigure} 
	\begin{subfigure}{.45\textwidth}
		\tikzsetnextfilename{nlpdatabetatol}%
		\resizebox{\linewidth}{!}{\begin{tikzpicture}
\begin{axis}[
height=8cm,
width=12cm,
xmode = log,
cycle list name=clusters,
grid=both,
grid style={line width=.1pt, draw=gray!20},
major grid style={line width=.2pt,draw=gray!60},
minor tick num=5,
xlabel=$\epsilon$,ylabel=Speed-up,
legend columns=1,
legend pos=north west,
legend style={font=\small}
]
\addplot+ table[x=tol, y=b1] {Tikz/nlpdata_beta_tol.dat};
\addplot+ table[x=tol, y=b2] {Tikz/nlpdata_beta_tol.dat};
\addplot+ table[x=tol, y=b3] {Tikz/nlpdata_beta_tol.dat};
\addplot+ table[x=tol, y=b4] {Tikz/nlpdata_beta_tol.dat};
\legend{$\beta=1e-4$,$\beta=1e-3$,$\beta=1e-2$,$\beta=1e-1$}
\end{axis}
\end{tikzpicture}   }
		\caption{nlpdata}
		\label{fig:nlpdata_betatol}
	\end{subfigure}  
	\caption{ The speed-up of our two-level preconditioner using two iterations of CG as pre-smoother and Leader-Follower clustering with adjusted squareroot average interpolation with respect to CG for solving a regularized least squares linear system in function of the required tolerance $\epsilon$ for different regularization parameters $\beta$  for news20, E2006, SNP2 and nlpdata.}
	\label{fig:betatol}
\end{figure}

Figure \ref{fig:betatol} additionally shows that the speed-up depends on the required tolerance $\epsilon$. As can be seen in previous Figure \ref{fig:clustering} and \ref{fig:smoothing} the residual of our two-level preconditioner initially decays fast but does slow down. As a result the speed-up for the smaller tolerances can be less than for the larger tolerance. In practice however applications with required relative errors smaller than $10^{-6}$ are seldom.  

\section{Conclusion} \label{sect:conclusion}
The two-level preconditioner described here finds its application in machine learning. We have shown that it is possible to approximate the principal eigenvectors on a coarser level by means of clustering the columns of the feature matrix. This coarser level is used in a two-level preconditioner to accelerate Flexible CG. The efficiency
of our preconditioner is data dependent and clustering algorithms perform differently for distinct data sets. Using general clustering algorithms such as K-means++ and Leader-Follower, we have shown that the number of iterations reduces dramatically for several data sets. The setup time of the clustering algorithms can be significantly but can be reused when multiple system solves are required. The setup time can also be further reduced by only approximately calculating the distance to other features by using locality-sensitive-hashing \cite{gionis1999similarity,andoni2015optimal} . We showed results for clustering in the column space (features) of $X$. It is also possible to additionally cluster in the row space (samples), reducing
the execution time of SVD on the coarsest level.

Using the two-level preconditioner to accelerate FCG, speed-up was achieved for almost all data sets. The efficiency of the preconditioner was shown to be data dependent. Even modest relative speed-up can result in large absolute
gain. Even with modest speed-up, the number of iterations is significantly reduced resulting in a more robust iterative
solver. Using more than two-levels results in even less iterations but does increase the execution time.

\section{Acknowledgements}	
The authors would like to thank the reviewers for their comments and suggestions. This work was supported Research Foundation - Flanders (FWO, No. G079016N). We would additionally like to thank David Vanavermaete for generating the SNP data sets. 	

\bibliography{references.bib} 
\end{document}